\documentclass{amsart}

\DeclareSymbolFont{AMSb}{U}{msb}{m}{n}
\DeclareSymbolFontAlphabet{\Bb}{AMSb}

\usepackage{amssymb}
\usepackage{graphicx}
\usepackage{yhmath}
\usepackage{mathdots}
\usepackage{MnSymbol}
\usepackage[cp1250]{inputenc}
\usepackage{bbm}
\usepackage[normalem]{ulem}

\usepackage{tikz}
\usepackage{pgf,pgfplots,pgfarrows,pgfnodes,pgfautomata,pgfheaps,pgfshade}
\usetikzlibrary{intersections,calc,positioning}
\usetikzlibrary{shapes.geometric,fit}
\usetikzlibrary{decorations,arrows,decorations.pathmorphing,backgrounds,positioning,fit,petri,shadows}
\usetikzlibrary{3d}

\DeclareMathOperator*{\ch}{ch}
\DeclareMathOperator*{\wt}{wt}
\DeclareMathOperator*{\prt}{part}

\newtheorem{theorem}{Theorem}[section]

\newtheorem{lemma}[theorem]{Lemma}

\newtheorem{example}[theorem]{Example}
\newtheorem{remark}[theorem]{Remark}

\numberwithin{equation}{section}
\author{Mirko Primc}
\address[M. Primc]{Faculty of Science,  University of Zagreb,  Zagreb, Croatia}
\email{primc@math.hr}
\author{Goran Trup\v cevi\' c}
\address[G. Trup\v cevi\' c]{Faculty of Teacher Education,  University of Zagreb,  Zagreb, Croatia}
\email{goran.trupcevic@ufzg.hr}

\begin{document}
\title{Linear independence for $C_\ell^{(1)}$ by using $C_{2\ell}^{(1)}$}

\begin{abstract}
In this note we prove linear independence of the combinatorial spanning set for standard $C_\ell^{(1)}$-module $L(k\Lambda_0)$ by establishing a connection with the combinatorial basis of Feigin-Stoyanovsky's type subspace $W(k\Lambda_0)$ of $C_{2\ell}^{(1)}$-module $L(k\Lambda_0)$. It should be noted that the proof of linear independence for the basis of $W(k\Lambda_0)$ is obtained by using simple currents and intertwining operators in the vertex operator algebra $L(k\Lambda_0)$.
\end{abstract}
\maketitle
 
\section{Introduction}

Let ${\mathfrak g}$ be a simple Lie algebra of type $C_\ell$, $\ell\geq2$, and
$$
\hat{\mathfrak g} =\coprod_{j\in\mathbb Z}{\mathfrak g}\otimes
t^{j}+\mathbb C c
$$
the associated affine Lie algebra of type $C_\ell^{(1)}$. For a basis $B_\ell$ of ${\mathfrak g}$ (see Subsection 2.4) we have the basis $\bar B_\ell\cup\{c\}$ of $\hat{\mathfrak g}$, where
$$
\bar B_\ell=\{x(n)\mid x\in B_\ell, n\in\mathbb Z\}
$$
and, as usual, $b(n)=b\otimes t^{n}$. In \cite{PS2} it is conjectured that the spanning set consisting of monomial vectors
\begin{equation}\label{E: u(pi)}
u(\pi)v_{k\Lambda_0}=\left(\prod_{b\in B_\ell, n<0}b(n)^{m_{b(n)}}\right)v_{k\Lambda_0},
\end{equation}
where the exponents $\pi=(m_{b(n)}\mid b\in B_\ell, n<0)$ are $L_{C_\ell^{(1)}}(k\Lambda_0)$-admissible, is a basis of the vacuum standard $\hat{\mathfrak g}$-module $L_{C_\ell^{(1)}}(k\Lambda_0)$ of the given level $k$, with a highest weight vector $v_{k\Lambda_0}$.

In this note we prove this conjecture by using the construction \cite{BPT} of a basis of Feigin-Stoyanovsky's subspace $W_{C_\ell^{(1)}}(k\Lambda_0)\subset L_{C_\ell^{(1)}}(k\Lambda_0)$, consisting, for certain $B_\ell^1\subset B_\ell$, of monomial vectors
\begin{equation*}
w(\pi)v_{k\Lambda_0}=\left(\prod_{b\in B_\ell^1, n<0}b(n)^{m_{b(n)}}\right)v_{k\Lambda_0},
\end{equation*}
where the exponents $\pi=(m_{b(n)}\mid b\in B_\ell^1, n<0)$ are $W_{C_\ell^{(1)}}(k\Lambda_0)$-admissible. As noted in \cite{P3}, there is a bijection
\begin{equation*}
B_\ell\longleftrightarrow B_{2\ell}^1
\end{equation*}
such that the corresponding combinatorial conditions of $L_{C_\ell^{(1)}}(k\Lambda_0)$-admissible $\pi$ and $W_{C_{2\ell}^{(1)}}(k\Lambda_0)$-admissible $\pi$ coincide
(so from here on we use only the term $(k\Lambda_0)$-admissible).

 So the key idea for the proof of linear independence of the set of monomial vectors (\ref{E: u(pi)}) is to embed Lie algebra ${\mathfrak g}$ of type $C_\ell$ into ${\mathfrak g}'$ of type $C_{2\ell}$ together with its standard module
$$
L_{C_{\ell}^{(1)}}(k\Lambda_0)\subset L_{C_{2\ell}^{(1)}}(k\Lambda_0)
\supset W_{C_{2\ell}^{(1)}}(k\Lambda_0),
$$
and then, by using inner derivations $T_{\underline 1},\dots, T_{\underline \ell}$ of ${\mathfrak g}'$, connect the sets $B_\ell$ and $B_{2\ell}^1$ and monomials $u(\pi)$ and $w(\pi)$ in the universal enveloping algebra $U(\widehat{\mathfrak g'})$:
$$
T_{\underline 1}^{m_{\underline 1}(\pi)}\dots, T_{\underline \ell}^{m_{\underline \ell}(\pi)}\colon u(\pi)\to w(\pi).
$$

For $k=1$ linear independence of the set of monomial vectors (\ref{E: u(pi)}) is proved in \cite{PS1}, \cite{DK} and \cite{R}.

It should be noted that the proof of linear independence for the basis of $W(k\Lambda_0)$ in \cite{BPT} is obtained by using simple currents and intertwining operators for the vertex operator algebra $L_{C_{\ell}^{(1)}}(k\Lambda_0)$. This method was used in \cite{CLM1} and \cite{CLM2} to find Rogers-Ramanujan and Rogers-Selberg recursions for the characters of Feigin-Stoyanovsky's subspaces of standard $A_1^{(1)}$-modules, and it seems that this approach is very well-suited for inductive arguments and recursions (cf. \cite{B, J1, J2, P1, P2, T1, T2}).
However, it is not clear whether this approach can yield a proof of linear independence of bases for all $L_{C_{\ell}^{(1)}}(\Lambda)$ which are implicitly conjectured in \cite{CMPP}. J. Dousse and I. Konan proved in \cite{DK} unspecialized character formula for all level $1$ standard modules $L_{C_{\ell}^{(1)}}(\Lambda)$ by using crystal base theory, and M. Russel proved in \cite{R} Andrews-Bressoud type identities by solving recursions which involved all level $1$ principally specialized characters of $L_{C_{\ell}^{(1)}}(\Lambda)$, but not only them(!?). Groebner basis approach is used in \cite{PS1} in $k=1$ case, but for higer levels this most general approach gets to be very difficult (cf. \cite{PS3}).

In this note we are interested in vacuum standard modules for affine Lie algebra of type  $C_\ell^{(1)}$, $\ell\geq 2$, but most of our considerations hold as well for $\ell=1$ when $C_1^{(1)}\cong A_1^{(1)}$. In particular, we have Theorem \ref{T:theorem} for $l=1$, which was proved by different methods in \cite{MP} and \cite{FKLMM}. Our proof is, in a way, inspired by E. Feigin's degeneration procedure in \cite{F}.

\section{Standard modules for affine Lie algebras}

\subsection{Affine Lie algebras}
Let ${\mathfrak g}$ be a simple complex Lie algebra, $\mathfrak h$
a Cartan subalgebra, $R$ the corresponding root system and $\theta$  the maximal root with respect to some fixed basis of the root system. For each root $\alpha$ we fix a root vector $x_\alpha$ in $\mathfrak g$. Via a symmetric invariant bilinear form $\langle \ , \ \rangle$ on ${\mathfrak g}$ we identify $\mathfrak h$ and $\mathfrak h^*$ and we assume
that $\langle \theta , \theta \rangle=2$.  Set
$$
\hat{\mathfrak g} =\coprod_{j\in\mathbb Z}{\mathfrak g}\otimes
t^{j}+\mathbb C c, \qquad \tilde{\mathfrak g}=\hat{\mathfrak
g}+\mathbb C d.
$$
Then $\tilde{\mathfrak g}$ is the associated untwisted affine
Kac-Moody Lie algebra (cf. \cite{K}) with the commutator
$$
[x(i),y(j)]=[x,y](i+j)+i\delta_{i+j,0}\langle x,y\rangle c.
$$
Here, as usual, $x(i)=x\otimes t^{i}$ for $x\in{\mathfrak g}$ and
$i\in\mathbb Z$, $c$ is the canonical central element, and
$[d,x(i)]=ix(i)$.  We identify ${\mathfrak g}$ and
${\mathfrak g}\otimes 1$. In the later sections, we will often use the notation $x_n$ for $x(n)$.
Let $B$ be a basis of $\mathfrak{g}$ consisting of root vectors and elements of $\mathfrak{h}$. Set
$$
\bar B=\{x(n)\mid x\in B, n\in\mathbb Z\}.
$$
Then $\bar B\cup\{c\}$ is a basis of $\hat{\mathfrak{g}}$.
A given linear order $\preceq$ on $B$ we extend  to a linear order on $\bar B$ by
\begin{equation}\label{E:order on bar B}
x(n)\prec y(m)\quad\text{iff}\quad n<m \ \text{or} \  n=m \ \text{and} \   x\prec y.
\end{equation}

\subsection{$\mathbb Z$-gradings of $\mathfrak g$ by miniscule coweights}
Let $R\sp\vee$ be the dual root system of $R$ and fix a minuscule coweight $\omega \in P(R^\vee)$. Set $\Gamma = \{\,\alpha \in R \mid \omega(\alpha\,) = 1\}$ and
$$
{\mathfrak g}_0 = {\mathfrak h} + \sum_{\omega(\alpha)=0}\,
{\mathfrak g}_\alpha, 
\quad  {\mathfrak g}_{\pm1} = \sum_{\alpha \in \pm \Gamma}\, {\mathfrak
g}_\alpha,\quad
 \hat{\mathfrak{g}}_0=\mathfrak{g}_0 \otimes \mathbb{C}[t,t^{-1}]\oplus \mathbb{C} c, \quad\hat{\mathfrak{g}}_{\pm1}=\mathfrak{g}_{\pm1}\otimes\mathbb{C}[t,t^{-1}].
$$
Then on $\mathfrak g$ and $\hat{\mathfrak{g}}$ we have $\mathbb{Z}$-gradings
$$
\mathfrak{g} = \mathfrak{g}_{-1} + \mathfrak{g}_0 + \mathfrak{g}_1\quad\text{and}\quad \hat{\mathfrak{g}} = \hat{\mathfrak{g}}_{-1} + \hat{\mathfrak{g}}_0 + \hat{\mathfrak{g}}_1 .
$$ 
Let ${\mathfrak h}_0'\subset{\mathfrak h}$ be a Cartan subalgebra of ${\mathfrak g}_0'=[{\mathfrak g}_0,{\mathfrak g}_0]$. Note that  ${\mathfrak g}_0 ={\mathfrak g}_0' +\mathbb C\omega$ is a reductive Lie algebra,  ${\mathfrak g}_{\pm1}$ are commutative subalgebras of $\mathfrak g$ and $\mathfrak g _0$-modules, and that
$\hat{\mathfrak{g}}_{\pm1}$ are commutative subalgebras of $\hat{\mathfrak{g}}$ and  $\hat{\mathfrak{g}}_0$-modules.

\subsection{Standard modules and Feigin-Stoyanovsky's type subspaces}
As usual, we denote by $\Lambda_0, \dots, \Lambda_{\ell}$ the fundamental weights of $\tilde{\mathfrak{g}}$. For a given integral domi\-nant weight $\Lambda = k_0\Lambda_0 + \cdots + k_{\ell} \Lambda_{\ell}$ we denote by $L(\Lambda)$ or $L_{\tilde{\mathfrak{g}}}(\Lambda)$ the standard (i.e. integrable highest weight) $\tilde{\mathfrak{g}}$-module with the highest weight $\Lambda$, by $v_{\Lambda}$ a fixed highest weight vector and by  $k=\Lambda(c)$ the level of the module (cf. \cite{K}). In this note we shall be interested mostly in the vacuum standard modules $L(k\Lambda_0)$.
\smallskip

For a standard $\tilde{\mathfrak{g}}$-module $L(\Lambda)$ and a given $\mathbb{Z}$-grading on $\hat{\mathfrak{g}}$ as above, we define the corresponding Feigin-Stoyanovsky's type subspace $W(\Lambda)$ or $W_{\tilde{\mathfrak{g}}}(\Lambda)$ as 
$$
 W(\Lambda) = U(\hat{\mathfrak{g}}_1)  v_{\Lambda}\subset
L(\Lambda).
 $$
Note that $U(\hat{\mathfrak{g}}_1) \cong S(\hat{\mathfrak{g}}_1) $ since $\hat{\mathfrak{g}}_1$ is commutative.

\subsection{A basis of  $\hat{\mathfrak g}$ of the type $C_\ell^{(1)}$}
Let ${\mathfrak g}$ be a simple Lie algebra  of the type $C_\ell$ and
 \ $\mathfrak g=\mathfrak n_-+{\mathfrak h}+{\mathfrak n_+}$ \  a Cartan decomposition. We identify the root system
 $$
 R=\{\pm\epsilon_i\pm\epsilon_j\mid i,j=1,\dots,\ell\}\backslash\{0\},
 $$
 where $\{\epsilon_i\mid i=1,\dots,\ell\}$ is the canonical basis in $\mathbb R^\ell$. For $1\leq a, b\leq \ell$ fix root vectors in $\mathfrak g$ denoted as
\begin{itemize}
\item $ab$  \quad for the root $\epsilon_a+\epsilon_b$, \quad $a\leq b$,
\item $\underline{a}\underline{b}$  \quad for the root $-\epsilon_a-\epsilon_b$,  \quad  $a\geq b$,
\item ${a}\underline{b}$  \quad for the root $\epsilon_a-\epsilon_b$,  \quad  $a\neq b$,
\item ${a}\underline{a}=h_a$  \quad (a simple coroot in $\mathfrak h$).
\end{itemize}
These vectors form a weight basis $B_\ell=B$ of ${\mathfrak g}$. Set 
\begin{equation}\label{E:ordered indices} 1\succ\dots\succ\ell\succ\underline{\ell}\succ\dots\succ\underline{1}.
\end{equation}
We write  $B_\ell$ in a triangular scheme with the element
$ab\in B_\ell$ in $a$-th column and $b$-th row. For $\ell=3$ we have $B_3$

	\begin{center}
		\begin{tikzpicture}[scale=.5]
			\begin{scope}				
				\foreach \x in {1,...,3}{
						\fill[black!10] ({\x-.5},{\x-.5}) rectangle ({\x+.5},{\x+.5});
				}
				\foreach \y in {1,...,3}{
					\foreach \x in {\y,...,3}{
						\node at ({\y},{7-\x}){\y\x} ;
					}
					\foreach \x in {3,...,1}{
						\node at ({\y},{\x}){\y \underline{\x}} ;
					}
				}
				\foreach \y in {3,...,1}{
					\foreach \x in {\y,...,1}{
						\node at ({7-\y},{\x}){\underline{\y}\underline{\x}} ;
					}
				}
			\end{scope}			
		\end{tikzpicture}
	\end{center}
On the picture, the shaded part corresponds to the Cartan subalgebra $\mathfrak h$, the upper-left triangle to positive root vectors, ${\mathfrak n}_+$, and the lower-right triangle to negative root vectors, ${\mathfrak n}_-$. Note that the hypotenuse corresponds to the long roots and that $11$ is the root vector for the highest root $\theta=2\epsilon_1$ of $\mathfrak{g}$.

We write $\bar{B}_\ell=\bar{B}$ for the affine Lie algebra $\hat{\mathfrak g}$ of the type $C_\ell^{(1)}$. In order to describe the combinatorial monomial basis of $L(k\Lambda_0)$ we need a subset $\bar{B}_\ell^{<0}\subset\bar{B}$,
$$
\bar{B}_\ell^{<0}=
{B}_\ell\otimes t^{-1} \ {\cup} \ {{B}_\ell\otimes t^{-2}} \  {\cup}  \  {B}_\ell\otimes t^{-3} \ {\cup} \dots ,
$$
written as the array composed of a sequence of triangles as in \cite{T3}. For $\ell=2$ we have the array $\bar{B}_2^{<0}$

	\begin{center}

\begin{tikzpicture}[xscale=.7, yscale =.5]
		\begin{scope}				
			\foreach \y in {1,...,2}{
				\foreach \x in {\y,...,2}{
					\node at ({\x},{5-\y}){\y\x$\scriptstyle _{-1}$} ;
				}
				\foreach \x in {2,...,1}{
					\node at ({5-\x},{5-\y}){\y \underline{\x}$\scriptstyle _{-1}$} ;
				}
			}
			\foreach \y in {1,...,2}{
				\foreach \x in {\y,...,1}{
					\node at ({5-\x},{\y}){\underline{\y}\underline{\x}$\scriptstyle _{-1}$} ;
				}
			}
		\end{scope}			
		\begin{scope} [xshift=4cm]				
				\foreach \x in {1,...,4}{
					\foreach \y in {\x,...,4}{
						\fill[black!10] ({\x-.5},{\y-\x+.5}) rectangle ({\x+.5},{\y-\x+1.5});
					}
			}
			\foreach \y in {1,...,2}{
				\foreach \x in {\y,...,2}{
					\node at ({\y},{5-\x}){\y\x$\scriptstyle _{-2}$} ;
				}
				\foreach \x in {2,...,1}{
					\node at ({\y},{\x}){\y \underline{\x}$\scriptstyle _{-2}$} ;
				}
			}
			\foreach \y in {2,...,1}{
				\foreach \x in {\y,...,1}{
					\node at ({5-\y},{\x}){\underline{\y}\underline{\x}$\scriptstyle _{-2}$} ;
				}
			}
		\end{scope}			
		\fill[black!10] (8.5,.5) -- ++(1,0)--++(0,-1) -- ++(1,0)--++(0,-1) -- ++ (-2,0) -- cycle;
		\node at (5,0){11$\scriptstyle _{-3}$} ;
		\node at (6,0){12$\scriptstyle _{-3}$} ;
		\node at (7,0){1\underline{2}$\scriptstyle _{-3}$} ;
		\node at (8,0){1\underline{1}$\scriptstyle _{-3}$} ;
		\node at (9,0){11$\scriptstyle _{-4}$} ;
		\node at (6,-1){$\ddots$} ;
		\node at (7,-1){$\ddots$} ;
		\node at (8,-1){$\ddots$} ;
		\node at (9,-1){$\ddots$} ;
		\node at (10,-1){$\ddots$} ;
	\end{tikzpicture}
\end{center}
where $ab_{-n}=ab(-n)=ab\otimes t^{-n}$. We shall write this array of root vectors $\bar{B}_\ell^{<0}$ rotated for $\pi/4$, denoting only the position $\circ$ of an element $x\in \bar{B}_\ell^{<0}$:

	
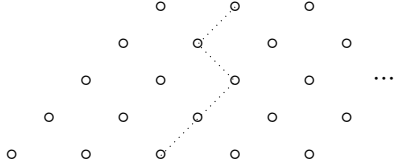
\begin{center}
	\begin{tikzpicture}[xscale=.7, yscale =.7, rotate=45]
		\begin{scope}				
			\foreach \y in {1,...,3}{
				\foreach \x in {1,...,5}{
					\node at ({\x+\y-1},{1-\y}){$\circ$};
				}
			}
				\foreach \x in {1,...,4}{
	\node at ({\x+4-1},{1-4}){$\circ$};
}
				\foreach \x in {1,...,2}{
	\node at ({\x+5-1},{1-5}){$\circ$};
}

\node at (7,-4) {$\cdots$};
			\draw[dotted] 	(6,-1) -- ++(-1,0) -- ++(0,-1) -- ++(-2,0);
		\end{scope}			
	\end{tikzpicture}
\end{center}
Of course, it is easy to see that the second element in the top row is $x=22_{-2}$. A downward path $\mathcal Z$ in the array $\bar{B}_\ell^{<0}$ is a sequence (or a subset) with an element in the top row followed by an adjacent element in the second row, and so on all the way to an element in the bottom row. An example of a downward path is denoted with the dashed line.

For $\ell>2$ the array $\bar{B}_\ell^{<0}$ is composed of a sequence of triangles so that for triangles ${B}_\ell\otimes t^{-n}$ and ${B}_\ell\otimes t^{-n-1}$ the adjacent sides are
$$
({1}\underline{1}_{-n}, \dots, {\ell}\underline{1}_{-n}, \underline{\ell}\underline{1}_{-n}, \dots, \underline{1}\underline{1}_{-n})\quad\text{and}\quad
({1}{1}_{-n-1}, \dots, 1{\ell}_{-n-1}, 1\underline{\ell}_{-n-1}, \dots,{1}\underline{1}_{-n-1})
$$
and the weight difference between the corresponding points is given by 
$$
\text{wt\,}({1a}_{-n-1})=-\alpha_0+\text{wt\,}({a}\underline{1}_{-n}),
$$
where $\alpha_0=-\theta+\delta$ is the simple root and $\delta$ is the imaginary root for the affine Lie algebra $\tilde{\mathfrak g}$. The array of root vectors $\bar{B}_\ell^{<0}$ rotated for $\pi/4$ has $2\ell+1$ rows and through each point in the first row there is $2^{2\ell}$ downward paths.

\subsection{Combinatorial spanning set of $C_\ell^{(1)}$-modules $L(k\Lambda_0)$} By using (\ref{E:ordered indices}) we define a lexicographical order $\preceq$ on $B_\ell$ and we extend it to $\bar B_\ell$ by (\ref{E:order on bar B}).
A monomial
\begin{equation}\label{E:colored partition pi}
\pi=\prod_{b(j)\in \bar{B}_\ell^{<0}}b(j)\sp{m_{b(j)}}\in S(\tilde{\mathfrak{g}})
\end{equation}
can be interpreted as a colored partition $\pi$: for $m_{b(j)}>0$ we say that $b(j)$ is a part of degree $|b(j)|=j$ and color $b\in B_\ell$ which appears in the partition $m_{b(j)}$ times. We define the degree and length of a colored partition $\pi$ as
$$
|\pi| = \sum_{b(j)\in \bar{B}_\ell^{<0}}j\cdot{m_{b(j)}},\quad 
\ell(\pi)=\sum_{b(j)\in \bar{B}_\ell^{<0}}{m_{b(j)}}.
$$
For a positive integer $k$ we say that a monomial (colored partition) $\pi$ is $k\Lambda_0$-admissible if for every downward path $\mathcal Z$ in the array $\bar{B}_\ell^{<0}$ we have
\begin{equation*}
\sum_{b(j)\in \mathcal Z}{m_{b(j)}}\leq k.
\end{equation*}
For a monomial $\pi$ of length $s$,
$$
\pi=x_1 x_2 \dots x_s,\quad x_1\preceq x_2 \preceq \dots \preceq x_s,
$$
we define the corresponding monomial in the enveloping algebra
\begin{equation}\label{E:monomial u(pi)}
u(\pi)=x_1 x_2 \dots x_s\in U(\tilde{\mathfrak g})
\end{equation}
and the coresponding monomial vector in the standard module
\begin{equation}\label{E:monomial vectors}
u(\pi) v_{k\Lambda_0}\in L(k\Lambda_0).
\end{equation}
It is proved in \cite{PS2} that the set of monomial vectors $u(\pi) v_{k\Lambda_0}$ with $k\Lambda_0$-admissible $\pi$ is a spanning set of $L(k\Lambda_0)$. Moreover, in \cite{PS1} it is proved that this is a basis for $k=1$ and it is conjectured that this spanning set is a basis for all integer $k>1$. The main result of this paper is

\begin{theorem}\label{T:theorem}
For a positive integer $k$ the set of monomial vectors $u(\pi) v_{k\Lambda_0}$ with $k\Lambda_0$-admissible $\pi$ is a basis of $L_{C_{\ell}^{(1)}}(k\Lambda_0)$.
\end{theorem}

\begin{remark}\label{R: change of basis}
{\em	Let
	\begin{equation}\label{E:definition of part}
		\prt_{\ell, k}=
		\sum_{\pi\text{ is $k\Lambda_0$-admissible}}e^{\wt(\pi)}
	\end{equation}
	be the generating function for the number of $k\Lambda_0$-admissible partitions/monomials with parts in the array of root vectors $\bar B_\ell^{<0}$. Since the set of monomial vectors (\ref{E:monomial vectors}) with $k\Lambda_0$-admissible $\pi$ is a spanning set of $L_{C_{\ell}^{(1)}}(k\Lambda_0)$, we have
	\begin{equation}\label{E:spanning inequality}
		\prt_{\ell, k}\geq e^{-k\Lambda_0}\ch L_{C_{\ell}^{(1)}}(k\Lambda_0)
	\end{equation}
	(the inequalities are true in each weight component) and Theorem \ref{T:theorem} holds if and only if the equality holds, i.e., if
	\begin{equation}\label{E:lin. indep. inequality}
		\prt_{\ell, k}\leq e^{-k\Lambda_0}\ch L_{C_{\ell}^{(1)}}(k\Lambda_0).
	\end{equation}
	
	To prove inequality (\ref{E:lin. indep. inequality}) we ``slightly'' change the basis $B_\ell$ of $\mathfrak g$  by setting
	\begin{equation}\label{E:change of basis B}
		{1}\underline{1}=\epsilon_1, {2}\underline{2}=\epsilon_2, \dots, {\ell}\underline{\ell}=\epsilon_\ell
	\end{equation}
	and repeat the construction of monomial vectors $u(\pi)v_{k\Lambda_0}$ with $k\Lambda_0$-admissible $\pi$ with parts in the new array of root vectors (which we denote again as $\bar B_\ell^{<0}$). These new $k\Lambda_0$-admissible partitions/monomials have the same generating function $\text{part}_{\ell, k}$ since the ``new'' and ``old'' elements ${a}\underline{a}_n$ have the same weight $n\delta$, and the inequality (\ref{E:lin. indep. inequality}) is obtained  by proving linear independence of this set of vectors. We prove linear independence in  Section \ref{S: proof} by establishing a connection with the combinatorial basis of Feigin-Stoyanovsky's type subspace $W(k\Lambda_0)$ of $C_{2\ell}^{(1)}$-module $L(k\Lambda_0)$ constructed in \cite{BPT}.
}\end{remark}

\subsection{Combinatorial bases of Feigin-Stoyanovsky type subspaces for $C_{2\ell}^{(1)}$}\label{Subsection: combinatorial bases of FS subspaces}

For $\mathfrak g$ of the type $C_\ell$ there is (up to conjugation) only one
$\mathbb Z$-grading 
$$
\mathfrak{g} = \mathfrak{g}_{-1} + \mathfrak{g}_0 + \mathfrak{g}_1
$$
with $\mathfrak{h}\subset \mathfrak{g}_0$ and we choose the basis 
$$
B_\ell^1=\{ij\mid 1\leq i\leq j\leq\ell\}
$$
for $\mathfrak{g}_1$ organized in a triangular scheme. 
For $\ell=3$ we have the elements of $B_3^1$ in the upper shaded triangle:

	\begin{center}
	\begin{tikzpicture}[scale=.5]
		\begin{scope}				
			\foreach \x in {1,...,3}{
							\fill[black!10] (.5,{\x+2.5}) rectangle ({4-\x+.5},{\x+3.5});
				}
			\foreach \y in {1,...,3}{
				\foreach \x in {\y,...,3}{
					\node at ({\y},{7-\x}){\y\x} ;
				}
				\foreach \x in {3,...,1}{
					\node at ({\y},{\x}){\y \underline{\x}} ;
				}
			}
			\foreach \y in {3,...,1}{
				\foreach \x in {\y,...,1}{
					\node at ({7-\y},{\x}){\underline{\y}\underline{\x}} ;
				}
			}
		\end{scope}			
	\end{tikzpicture}
\end{center}
From now on, for fixed $\ell\geq 2$, we consider $B_{2\ell}^1$. For $\ell=3$ we have a basis $B_{6}^1$ of $\mathfrak{g}_1$ for $\mathfrak g$ of the type $C_{6}$:

	\begin{center}
	\begin{tikzpicture}[scale=.5]
		\begin{scope}				
			\foreach \y in {1,...,6}{
				\foreach \x in {\y,...,6}{
					\node at ({\y},{7-\x}){\y\x} ;
				}
			}
		\end{scope}			
	\end{tikzpicture}
\end{center}

In order to describe the combinatorial monomial basis of Feigin-Stoyanovsky's subspace $W(k\Lambda_0)$ for $C_{2\ell}^{(1)}$ we need a subset ${\bar B}_{2\ell}^{1,<0}\subset\bar{B}_{2\ell}$,
$$
{\bar B}_{2\ell}^{1,<0}=
{B}_{2\ell}^1\otimes t^{-1} \ {\cup} \ {{B}_{2\ell}^1\otimes t^{-2}} \  {\cup}  \  {B}_{2\ell}^1\otimes t^{-3} \ {\cup} \dots ,
$$
written as the array composed of a sequence of triangles. For $\ell=2$ we have the array $\bar{B}_4^{1,<0}$

	\begin{center}

	\begin{tikzpicture}[xscale=.7, yscale =.5]
		\begin{scope}				
	\foreach \x in {1,...,4}{
		\foreach \y in {\x,...,4}{
			\node at ({\y},{5-\x}){\x\y$\scriptstyle _{-1}$} ;
		}
	}
\end{scope}			
\begin{scope} [xshift=4cm]				
				\foreach \x in {1,...,4}{
	\foreach \y in {\x,...,4}{
		\fill[black!10] ({\x-.5},{\y-\x+.5}) rectangle ({\x+.5},{\y-\x+1.5});
	}
}
	\foreach \x in {1,...,4}{
		\foreach \y in {\x,...,4}{
			\node at ({\x},{5-\y}){\x\y$\scriptstyle _{-2}$} ;
		}
	}
\end{scope}			
		\fill[black!10] (8.5,.5) -- ++(1,0)--++(0,-1) -- ++(1,0)--++(0,-1) -- ++ (-2,0) -- cycle;
\node at (5,0){11$\scriptstyle _{-3}$} ;
\node at (6,0){12$\scriptstyle _{-3}$} ;
\node at (7,0){13$\scriptstyle _{-3}$} ;
\node at (8,0){14$\scriptstyle _{-3}$} ;
\node at (9,0){11$\scriptstyle _{-4}$} ;
		\node at (6,-1){$\ddots$} ;
		\node at (7,-1){$\ddots$} ;
		\node at (8,-1){$\ddots$} ;
		\node at (9,-1){$\ddots$} ;
		\node at (10,-1){$\ddots$} ;
	\end{tikzpicture}
\end{center}
where $ij_{-n}=ij(-n)=ij\otimes t^{-n}$. We shall write this array of root vectors $\bar{B}_{4}^{1,<0}$ rotated for $\pi/4$, denoting only the position $\circ$ of an element $x\in \bar{B}_{4}^{1,<0}$:

	
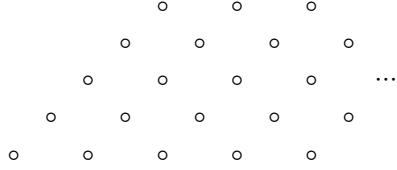
\begin{center}
	\begin{tikzpicture}[xscale=.7, yscale =.7, rotate=45]
		\begin{scope}				
			\foreach \y in {1,...,3}{
				\foreach \x in {1,...,5}{
					\node at ({\x+\y-1},{1-\y}){$\circ$};
				}
			}
			\foreach \x in {1,...,4}{
				\node at ({\x+4-1},{1-4}){$\circ$};
			}
			\foreach \x in {1,...,2}{
				\node at ({\x+5-1},{1-5}){$\circ$};
			}
			\node at (7,-4) {$\cdots$};
		\end{scope}			
	\end{tikzpicture}
\end{center}
The array $\bar{B}_{2\ell}^{1,<0}$ rotated for $\pi/4$ has $2\ell+1$ rows.
A downward path $\mathcal Z$ in the array $\bar{B}_{2\ell}^{1,<0}$ is defined as before,
and through each point in the first row there is $2^{2\ell}$ downward paths.

Like before, a monomial
$$
\pi=\prod_{b(j)\in \bar{B}_{2\ell}^{1,<0}}b(j)\sp{m_{b(j)}}\in S(\tilde{\mathfrak{g}}_1)
$$
can be interpreted as a colored partition $\pi$ with colors $b\in B_{2\ell}^1$. 
Formally the definition of $k\Lambda_0$-admissible colored partitions $\pi$ on the array $\bar{B}_{2\ell}^{1,<0}$ is the same as before.

Note that the Lie algebra $\mathfrak{g}_1$ is commutative, so the enveloping algebra is commutative, that is  $U(\mathfrak{g}_1)\cong S(\mathfrak{g}_1)$. It is proved in \cite{BPT} that the set of monomial vectors $u(\pi) v_{k\Lambda_0}$ with $k\Lambda_0$-admissible $\pi$ is a basis set of the Feigin-Stoyanovsky's subspace $W(k\Lambda_0)$.

\section{The proof of linear independence} \label{S: proof}

Let $\mathfrak g'$ be a simple Lie algebra of the type $C_{2\ell}$.
For any subset $\{i_1,\dots,i_\ell\}\subset \{1,\dots,2\ell\}$, a subalgebra  generated by root vectors
$$\{ab \ | \ a,b\in \{i_1,\dots,i_\ell,\underline{i_\ell},\dots,\underline{i_1}\},a\succeq b  \}$$
is a Lie algebra of the type $C_\ell$ (recall the change of basis from Remark \ref{R: change of basis}). This induces inclusions of the corresponding affine Lie algebra $C_\ell^{(1)}$ and its vacuum standard module
$$L_{C_\ell^{(1)}}(k\Lambda_0)\subset L_{C_{2\ell}^{(1)}}(k\Lambda_0).$$

In the remaining considerations, we fix a subset $\{1,\dots,\ell\}\subset \{1,\dots,2\ell\}$ and the corresponding subalgebra $\mathfrak g\subset \mathfrak g'$; for $\ell=3$ we see on Figure 1 the shaded basis $B_{3}\subset B_{6}$ and the boxed basis $B_{6}^{1}\subset B_{6}$. We consider the basis $\bar{B}_\ell^{<0}$ as a subset of $\bar{B}_{2\ell}^{<0}$ and similarly for monomials in $\bar{B}_\ell^{<0}$ and $\bar{B}_{2\ell}^{<0}$ and the corresponding monomial vectors. 
\begin{figure}
	\caption{Bases $B_{6}$, shaded $B_3$ and boxed $B_{6}^{1}$}
	\begin{center}
		\begin{tikzpicture}[scale=.5]
			\begin{scope}				
				\foreach \y in {1,...,3}{
					\foreach \x in {\y,...,3}{
						\fill[black!10] ({\y-.5},{13-\x-.5}) rectangle ({\y+.5},{13-\x+.5});
					}
					\foreach \x in {3,...,1}{
						\fill[black!10] ({\y-.5},{\x-.5}) rectangle ({\y+.5},{\x+.5});
					}
				}
				\foreach \y in {3,...,1}{
					\foreach \x in {\y,...,1}{
						\fill[black!10] ({13-\y-.5},{\x-.5}) rectangle ({13-\y+.5},{\x+.5});
					}
				}
				\foreach \y in {1,...,6}{
					\foreach \x in {\y,...,6}{
						\node at ({\y},{13-\x}){\y\x} ;
					}
					\draw (\y-.5,13.5-\y) -- ++(1,0) -- ++(0,-1);
					\foreach \x in {6,...,1}{
						\node at ({\y},{\x}){\y \underline{\x}} ;
					}
				}
				\foreach \y in {6,...,1}{
					\foreach \x in {\y,...,1}{
						\node at ({13-\y},{\x}){\underline{\y}\underline{\x}} ;
					}
				}
				\draw	(.5,12.5) -- ++(0,-6) -- ++(6,0);
			\end{scope}			
		\end{tikzpicture}
	\end{center}
\end{figure}
The basic idea for the proof of linear independence is to shift colors of monomials in $\bar{B}_\ell^{<0}$ to $\bar{B}_{2\ell}^{1,<0}$ (see Figure 2). The same combinatorial description of the spanning set of the vacuum standard module $L_{C_\ell^{(1)}}(k\Lambda_0)$ from \cite{PS2} and  of the basis of Feigin-Stoyanovsky's subspace $W_{C_{2\ell}^{(1)}}(k\Lambda_0)$ from \cite{BPT} gives  linear independence of the first set. 

To shift colors from ${B}_\ell$ to ${B}_{2\ell}^{1}$, or from $\bar{B}_\ell$ to $\bar{B}_{2\ell}^{1}$, we use the adjoint action $T_{\underline \ell}, \dots, T_{\underline 1}$ by suitable elements of $\mathfrak{g}'$. Note that for $b\in B_{2\ell}$ and $n<0$ we have
\begin{equation}\label{E: T(bn)=(Tb)n}
T_{\underline a}\left(b_n\right)=\left(T_{\underline a}b\right)_n.
\end{equation}
\begin{figure}
	\caption{Shifting colors from $B_3$ to $B_{6}^{1}$}
	\begin{center}
		\begin{tikzpicture}[scale=.2]
			\begin{scope}
				\foreach \y in {1,...,3}{
					\foreach \x in {\y,...,3}{
						\fill[black!10] ({\y-.5},{13-\x-.5}) rectangle ({\y+.5},{13-\x+.5});
					}
					\foreach \x in {3,...,1}{
						\fill[black!10] ({\y-.5},{\x-.5}) rectangle ({\y+.5},{\x+.5});
					}
				}
				\foreach \y in {3,...,1}{
					\foreach \x in {\y,...,1}{
						\fill[black!10] ({13-\y-.5},{\x-.5}) rectangle ({13-\y+.5},{\x+.5});
					}
				}
				\foreach \y in {1,...,6}{
					\foreach \x in {\y,...,6}{
						\node at ({\y},{13-\x}){$\scriptstyle \circ$} ;
					}
					\draw (\y-.5,13.5-\y) -- ++(1,0) -- ++(0,-1);
					\foreach \x in {6,...,1}{
						\node at ({\y},{\x}){$\scriptstyle \circ$} ;
					}
				}
				\foreach \y in {6,...,1}{
					\foreach \x in {\y,...,1}{
						\node at ({13-\y},{\x}){$\scriptstyle \circ$} ;
					}
				}
				\draw	(.5,12.5) -- ++(0,-6) -- ++(6,0);
		\draw[->,decorate,decoration={snake,amplitude=.4mm,segment length=2mm,post length=1mm}] (8.5,6.5) -- ++(5,0);
			\end{scope}
		\begin{scope}[xshift=14cm]
		\foreach \y in {1,...,4}{
			\foreach \x in {\y,...,4}{
				\fill[black!10] ({\y-.5},{13-\x-.5}) rectangle ({\y+.5},{13-\x+.5});
			}
			\foreach \x in {2,...,1}{
				\fill[black!10] ({\y-.5},{\x-.5}) rectangle ({\y+.5},{\x+.5});
			}
		}
		\foreach \y in {2,...,1}{
			\foreach \x in {\y,...,1}{
				\fill[black!10] ({13-\y-.5},{\x-.5}) rectangle ({13-\y+.5},{\x+.5});
			}
		}
		\foreach \y in {1,...,6}{
			\foreach \x in {\y,...,6}{
				\node at ({\y},{13-\x}){$\scriptstyle \circ$} ;
			}
			\draw (\y-.5,13.5-\y) -- ++(1,0) -- ++(0,-1);
			\foreach \x in {6,...,1}{
				\node at ({\y},{\x}){$\scriptstyle \circ$} ;
			}
		}
		\foreach \y in {6,...,1}{
			\foreach \x in {\y,...,1}{
				\node at ({13-\y},{\x}){$\scriptstyle \circ$} ;
			}
		}
		\draw	(.5,12.5) -- ++(0,-6) -- ++(6,0);
		\draw[->,decorate,decoration={snake,amplitude=.4mm,segment length=2mm,post length=1mm}] (8.5,6.5) -- ++(5,0);
	\end{scope}
		\begin{scope}[xshift=28cm]
	\foreach \y in {1,...,5}{
		\foreach \x in {\y,...,5}{
			\fill[black!10] ({\y-.5},{13-\x-.5}) rectangle ({\y+.5},{13-\x+.5});
		}
		\foreach \x in {1,...,1}{
			\fill[black!10] ({\y-.5},{\x-.5}) rectangle ({\y+.5},{\x+.5});
		}
	}
	\foreach \y in {1,...,1}{
		\foreach \x in {1,...,1}{
			\fill[black!10] ({13-\y-.5},{\x-.5}) rectangle ({13-\y+.5},{\x+.5});
		}
	}
	\foreach \y in {1,...,6}{
		\foreach \x in {\y,...,6}{
			\node at ({\y},{13-\x}){$\scriptstyle \circ$} ;
		}
		\draw (\y-.5,13.5-\y) --  ++(1,0) -- ++(0,-1);
		\foreach \x in {6,...,1}{
			\node at ({\y},{\x}){$\scriptstyle \circ$} ;
		}
	}
	\foreach \y in {6,...,1}{
		\foreach \x in {\y,...,1}{
			\node at ({13-\y},{\x}){$\scriptstyle \circ$} ;
		}
	}
	\draw	(.5,12.5) -- ++(0,-6) -- ++(6,0);
		\draw[->,decorate,decoration={snake,amplitude=.4mm,segment length=2mm,post length=1mm}] (8.5,6.5) -- ++(5,0);
\end{scope}
		\begin{scope}[xshift=42cm]
	\foreach \y in {1,...,6}{
		\foreach \x in {\y,...,6}{
			\fill[black!10] ({\y-.5},{13-\x-.5}) rectangle ({\y+.5},{13-\x+.5});
		}
	}
	\foreach \y in {1,...,6}{
		\foreach \x in {\y,...,6}{
			\node at ({\y},{13-\x}){$\scriptstyle \circ$} ;
		}
		\draw (\y-.5,13.5-\y) -- ++(1,0) -- ++(0,-1);
		\foreach \x in {6,...,1}{
			\node at ({\y},{\x}){$\scriptstyle \circ$} ;
		}
	}
	\foreach \y in {6,...,1}{
		\foreach \x in {\y,...,1}{
			\node at ({13-\y},{\x}){$\scriptstyle \circ$} ;
		}
	}
	\draw	(.5,12.5) -- ++(0,-6) -- ++(6,0);
\end{scope}
		\end{tikzpicture}
	\end{center}
\end{figure}

\begin{example}{\em 
For $\ell=3$ we first move the shaded elements of $B_3$ in $\underline 3$-row to $4$-row in $B_6^1$ with $T_{\underline 3}=\text{ad\,}(34)$ (see Figure 2):
$$
{\underline 3}{\underline 3}\mapsto 4{\underline 3},\quad
{\underline 3}{\underline 2}\mapsto 4{\underline 2},\quad
{\underline 3}{\underline 1}\mapsto 4{\underline 1}\quad\text{and}\quad
1{\underline 3}\mapsto 14,\quad 
2{\underline 3}\mapsto 24,\quad
3{\underline 3}\mapsto 34,\quad
4{\underline 3}\mapsto 44,
$$
and $T_{\underline 3}=\text{ad\,}(34)$ is zero on all other shaded elements together with newly shaded $4{\underline 2}$ and $4{\underline 1}$ (recall the change of basis of Cartan subalgebra from Remark \ref{R: change of basis}). Now we have shaded $4$-row in $B_6^1$ and last two rows with shaded $4\times2$ rectangle at the begining and $2\times2$ triangle at the end.
Next we move the shaded elements in $\underline 2$-row to $5$-row in $B_6^1$ with $T_{\underline 2}=\text{ad\,}(25)$:
$$
{\underline 2}{\underline 2}\mapsto 5{\underline 2},\quad
{\underline 2}{\underline 1}\mapsto 5{\underline 1}\quad\text{and}\quad
1{\underline 2}\mapsto 15,\quad 
2{\underline 2}\mapsto 25,\quad
3{\underline 2}\mapsto 35,\quad
4{\underline 2}\mapsto 45,\quad
5{\underline 2}\mapsto 55
$$
and $T_{\underline 2}=\text{ad\,}(25)$ is zero on all other shaded elements together with newly shaded $5{\underline 1}$. Now we have shaded $5$-row in $B_6^1$ and one last row with shaded $5\times1$ ``rectangle''  at the begining and $1\times1$ ``triangle'' at the end.
Finally we move the shaded elements in $\underline 1$-row to $6$-row with $T_{\underline 1}=\text{ad\,}(16)$:
$$
{\underline 1}{\underline 1}\mapsto 6{\underline 1}\quad\text{and}\quad
1{\underline 1}\mapsto 16,\quad 
2{\underline 1}\mapsto 26,\quad
3{\underline 1}\mapsto 36,\quad
4{\underline 1}\mapsto 46,\quad
5{\underline 1}\mapsto 56,\quad
6{\underline 1}\mapsto 66.
$$

However, there is a drawback in this procedure -- the role of elements ${3}{\underline 3}=h_3$, ${2}{\underline 2}=h_2$ and ${1}{\underline 1}=h_1$ in $C_3\subset C_6$ while shifting colors with $T_{\underline 3}=\text{ad\,}(34)$ or $T_{\underline 2}=\text{ad\,}(25)$. We have 
\begin{eqnarray*}
& & \hspace{-4ex}	T_{\underline 3}:
	\quad {3}{\underline 3}\mapsto \langle \epsilon_3, \epsilon_3+\epsilon_4\rangle 34=34, 
	\quad {2}{\underline 2}\mapsto \langle \epsilon_2-\epsilon_3, \epsilon_3+\epsilon_4\rangle 34=-34, \\
	& &\quad {1}{\underline 1}\mapsto \langle \epsilon_1-\epsilon_2, \epsilon_3+\epsilon_4\rangle 34=0, \\ 
	& &  \hspace{-4ex}	T_{\underline 2}:
\quad {3}{\underline 3}\mapsto \langle \epsilon_3, \epsilon_2+\epsilon_5\rangle 25=0, 
\quad {2}{\underline 2}\mapsto \langle \epsilon_2-\epsilon_3, \epsilon_2+\epsilon_5\rangle 25=25, \\
& & 
\quad {1}{\underline 1}\mapsto \langle \epsilon_1-\epsilon_2, \epsilon_2+\epsilon_5\rangle 25=-25,  
\end{eqnarray*}
which could be difficult to handle. For this reason in (\ref{E:change of basis B}) we introduce elements ${3}{\underline 3}=\epsilon_3$, ${2}{\underline 2}=\epsilon_2$ and ${1}{\underline 1}=\epsilon_1$ so that the action of $T_{\underline 3}$ and $T_{\underline 2}$ on ${3}{\underline 3}$, ${2}{\underline 2}$ and ${1}{\underline 1}$ is as described above:
\begin{eqnarray*}
	& &  \hspace{-4ex}	T_{\underline 3}:
	\quad {3}{\underline 3}\mapsto 34, 
	\quad {2}{\underline 2}\mapsto 0, 
	\quad {1}{\underline 1}\mapsto 0, \\ 
	& &  \hspace{-4ex}	T_{\underline 2}:
	\quad {3}{\underline 3}\mapsto 0, 
	\quad {2}{\underline 2}\mapsto 25,
	\quad {1}{\underline 1}\mapsto 0, \\
	& &  \hspace{-4ex}	T_{\underline 1}:
\quad {3}{\underline 3}\mapsto 0, 
\quad {2}{\underline 2}\mapsto 0,
\quad {1}{\underline 1}\mapsto 16.
\end{eqnarray*}
}\end{example}
In general, $B_{2\ell}$ is the triangle with vertices $11,1{\underline 1},{\underline 1}{\underline 1}$; $B_{2\ell}^1$ is the triangle with vertices $11,1(2\ell),(2\ell)(2\ell)$ -- we think of it as boxed elements; and $B_{\ell}$ is the union of two $\ell\times\ell$ triangles with vertices $11,1\ell,\ell\ell$ and ${\underline \ell}{\underline \ell},{\underline \ell}{\underline 1},{\underline 1}{\underline 1}$ and one $\ell\times\ell$ square with vertices 
$\ell{\underline \ell},1{\underline \ell},1{\underline 1},\ell{\underline 1}$
 -- we think of it as shaded elements.

For $1\leq a\leq \ell$ set
\begin{equation}\label{E:def Ta}
T_{\underline a} = \text{ad\,}(t_{\underline a})\,\colon \mathfrak g' \to \mathfrak g',\quad
t_{\underline a} =a (2\ell-a+1)\in B_{2\ell}\subset\mathfrak g' .
\end{equation}
We use inner derivations $T_{\underline a}$ to shift in steps colors in $B_{\ell}$ into colors in $B_{2\ell}^1$, row by row. 
First we apply $T_{\underline \ell}=\text{ad\,}(\ell(\ell+1))$ and move the vertical catheti of the lower triangle
$[{\underline \ell}{\underline \ell},{\underline \ell}{\underline 1}]$ onto the vertical segment
$[(\ell+1){\underline \ell},(\ell+1){\underline 1}]$, forming a new shaded $\ell\times(\ell+1)$ rectangle. Then we move with $T_{\underline \ell}=\text{ad\,}(\ell(\ell+1))$ the upper side of newly formed shaded rectangle
$[1{\underline \ell}, (l+1){\underline \ell}]$ into the $(\ell+1)$-row 
$[1(\ell+1), (\ell+1)(\ell+1)]$
of boxed triangle $B_{2\ell}^1$. The result is the shaded union of $(\ell+1)\times(\ell+1)$ triangle on the top, $(\ell+1)\times(\ell-1)$ rectangle at the right angle and $(\ell-1)\times(\ell-1)$ triangle on the right.

In $r<\ell$ steps we get the shaded union of $(\ell+r)\times(\ell+r)$ triangle on the top, $(\ell+r)\times(\ell-r)$ rectangle at the right angle and $(\ell-r)\times(\ell-r)$ triangle on the right, so we can proceed by applying  $T_{\underline {\ell-r}}=\text{ad\,}((\ell-r)(\ell+r+1))$.

Our intuitive considerations above we can rephrase with formulas, keeping in mind the change of basis from Remark \ref{R: change of basis}:
\begin{lemma}\label{L: action of Ta}
$$
\begin{aligned}
	T_{\underline a}\,\underline{a} \underline{b}_n& \in  \mathbb C^\times(2\ell-a+1) \underline{b}_n,\\  
\	T_{\underline a}\, c \underline{a}_n & \in  \mathbb C^\times  c (2\ell-a+1)_n,  
\end{aligned}
$$
for $1\leq b\leq a \leq \ell$, $1\leq c \leq 2\ell-a+1 $. In the first case, there is a shift from the $\underline{a}$-column to the $(2\ell-a+1)$-column, while in the second case, there is a shift from the $\underline{a}$-row to the $(2\ell-a+1)$-row. Moreover
$$
\begin{aligned}
	T_{\underline a}^2\, c \underline{a}_n & =  0,  \\
\	T_{\underline a}^2\, \underline{a} \underline{b}_n & = 0 \ \text{for}\  b<a, \\
\	T_{\underline a}^3\, \underline{a} \underline{a}_n & =  0, \\
\   T_{\underline a}\,  ef_n & =  0  \  \text{for}\  e,f \notin \{\underline{a},\underline{2\ell-a+1}\}. 	
\end{aligned}
$$
\end{lemma}
Since an inner derivation on $\widehat{\mathfrak g'}$ extends to a derivation on  $U(\widehat{\mathfrak g'})$, we have:
\begin{lemma}\label{L: action of powers Ta}
For a positive integer $m$, $1 \leq  a \leq \ell$, $1\leq b < a$ and $1\leq c \leq 2\ell-a+1 $, we have
$$
\begin{aligned}
&T_{\underline a}^{2m}\,\left(\underline{a} \underline{a}\right)^m 
\in  \mathbb C^\times \left((2\ell-a+1)(2\ell-a+1)\right)^m,\\  
&T_{\underline a}^{2m+1}\,\left(\underline{a} \underline{a}\right)^m=0,\\ 
&T_{\underline a}^{m}\,\left(\underline{a} \underline{b} \right)^m 
\in  \mathbb C^\times \left((2\ell-a+1)  \underline{b}\right)^m,\\  
&T_{\underline a}^{m+1}\,\left(\underline{a}  \underline{b} \right)^m=0\\
&T_{\underline a}^{m}\,\left(c \underline{a}\right)^m 
\in  \mathbb C^\times \left(c(2\ell-a+1)\right)^m,\\  
&T_{\underline a}^{m+1}\,\left(c \underline{a}\right)^m=0.
\end{aligned}
$$
\end{lemma}
\begin{proof}
By Lemma \ref{L: action of Ta} we have
$$
T_{\underline a}^{2}\,\left(\underline{a} \underline{a}\right)
\in  \mathbb C^\times (2\ell-a+1)(2\ell-a+1)\quad\text{and}\quad T_{\underline a}^{3}\,\left(\underline{a} \underline{a}\right)=0.
$$
The monomial $\left(\underline{a} \underline{a}\right)^m$ has $m$ factors $\underline{a} \underline{a}$, and if we act with the derivation $T_{\underline a}$ on each of this factors twice, we get something proportional to $\left((2\ell-a+1)(2\ell-a+1)\right)^m$, and if we act on some factor more than two times we get $0$. So the first two statements follow from the pigeonhole principle. For other four statements we argue similarly.
\end{proof}

For a monomial $\pi$ with variables in $\bar{B}_\ell^{<0}$ given by (\ref{E:colored partition pi}), and $1\leq a \leq \ell$, define
$$
\begin{gathered}
	m_{\underline{a}}(\pi) =  \sum_{b=1}^\ell \sum_{n<0} m_{b\underline{a}}(n)   +  \sum_{b=1}^{a-1} \sum_{n<0} m_{\underline{a}\underline{b}}(n) + 2 \sum_{n<0} m_{\underline{a} \underline{a}}(n) +  \sum_{b=a+1}^\ell \sum_{n<0} m_{\underline{b}\underline{a}}(n),\\
	M(\pi)  =  \sum_{a=1}^\ell m_{\underline{a}}(\pi).
\end{gathered}
$$
For $x\in\mathfrak g'\subset \widehat{\mathfrak g'}$ we have the inner derivation $T=\text{ad}\, x$ on  $U(\widehat{\mathfrak g'})$. Since $xv_{k\Lambda_0}=0$, for $u\in U(\widehat{\mathfrak g'})$ we have
$x(uv_{k\Lambda_0})=[x,u]v_{k\Lambda_0}=(Tu)v_{k\Lambda_0}$.
In particular, for $T_{\underline a}$ and $t_{\underline a}$ defined by (\ref{E:def Ta}) we have
\begin{equation*}
t_{\underline a}(uv_{k\Lambda_0})=(T_{\underline a}u)v_{k\Lambda_0}.
\end{equation*}
Let
$$
\begin{aligned}
t(\pi)&=t_{\underline 1}^{m_{\underline{1}}(\pi)}t_{\underline 2} ^{m_{\underline{2}}(\pi)} \cdots t_{\underline \ell}^{m_{\underline{\ell}}(\pi)}\\
&=1 (2\ell)^{m_{\underline{1}}(\pi)} 2 (2\ell-1)^{m_{\underline{2}}(\pi)} \cdots \ell (\ell+1)^{m_{\underline{\ell}}(\pi)},\\
T(\pi)&=T_{\underline 1}^{m_{\underline{1}}(\pi)}T_{\underline 2} ^{m_{\underline{2}}(\pi)} \cdots T_{\underline \ell}^{m_{\underline{\ell}}(\pi)}.
\end{aligned}
$$
\begin{lemma}\label{L: T(pi)u(pi)=w(pi)} For a monomial $\pi$ given by (\ref{E:colored partition pi}) we have
$$
t(\pi) \left(u(\pi)v_{k\Lambda_0}\right)= \left(T(\pi)u(\pi)\right)v_{k\Lambda_0}
=w(\pi')v_{k\Lambda_0},
$$
where $w(\pi')$ is a monomial in $\bar{B}^{1,<0}_{2\ell}\subset \bar{B}^{<0}_{2\ell}$ such that
$$
\begin{aligned}
	m_{ab(n)}(\pi') & = 	m_{ab(n)}(\pi), \\
	m_{a (2\ell-b+1)(n)}(\pi') & = m_{a\underline{b}(n)}(\pi),\\
	m_{(2\ell-a+1) (2\ell-b+1)(n)}(\pi') & = m_{\underline{a}\underline{b}(n)}(\pi),\\
	m_{a\underline{b}(n)}(\pi') & = 	0, \\
	m_{\underline{a}\underline{b}(n)}(\pi') & =  	0,
\end{aligned}
$$
for $1\leq a,b \leq \ell$.
\end{lemma}
\begin{proof}
Assume that the monomial $u(\pi)$ defined by (\ref{E:monomial u(pi)}) has at least one factor $b(j)^{m_{b(j)}}$, $m_{b(j)}>0$, with color $b$ from the $\underline{\ell}$-row of $B_\ell$.
To shift colors in the monomial $u(\pi)$, first act with $T_{\underline \ell}^{m_{\underline{\ell}}(\pi)}=\ell (\ell+1)^{m_{\underline{\ell}}(\pi)}$. 
This shifts colors from $\underline{\ell}$-column to $(\ell+1)$-column, and from  $\underline{\ell}$-row to $(\ell+1)$-row. Monomial may have factors 
$$
1\underline{\ell}(n)^{m_{1\underline{\ell}(n)}}, 
2\underline{\ell}(n)^{m_{2\underline{\ell}(n)}}, 
\dots, \ell\underline{\ell}(n)^{m_{\ell\underline{\ell}(n)}}, 
\underline{\ell}\underline{\ell}(n)^{m_{\underline{\ell}\underline{\ell}(n)}}
$$
with colors in $\ell$-row for several values of $n<0$, but because of (\ref{E: T(bn)=(Tb)n})
$T_{\underline \ell}$ acts only on colors of these factors. Arguing as in the proof of Lemma \ref{L: action of powers Ta}, we see we need to act 
$$
m'=\sum_{n<0} \sum_{b=1}^{\ell}m_{b\underline{\ell}(n)}
$$
times with $T_{\underline \ell}$ to move all factors with colors $1\underline{\ell},\dots,\ell\underline{\ell}$ to factors with colors $1(\ell+1),\dots,\ell(\ell+1)$, and acting with $T_{\underline \ell}^{m'+1}$ will produce $0$. Schematically we write
$$
T_{\underline \ell}^{m'}\colon \ 
1\underline{\ell}(n)^{m_{1\underline{\ell}(n)}},  
\dots, \ell\underline{\ell}(n)^{m_{\ell\underline{\ell}(n)}},\forall n
\rightsquigarrow
1(\ell+1)(n)^{m_{1\underline{\ell}(n)}},  
\dots, \ell(\ell+1)(n)^{m_{\ell\underline{\ell}(n)}},\forall n.
$$
We also see that we need to act
$$
m''=2\sum_{n<0}m_{\underline{\ell}\underline{\ell}(n)}
$$
times with $T_{\underline \ell}$ to move all factors with color $\underline{\ell}\underline{\ell}$ to factors with colors  $(\ell+1)(\ell+1)$, and acting with $T_{\underline \ell}^{m''+1}$ will produce $0$. Schematically we write
$$
T_{\underline \ell}^{m''}\colon \ 
\underline{\ell}\underline{\ell}(n)^{m_{\underline{\ell}\underline{\ell}(n)}}, \forall n
\rightsquigarrow
(\ell+1)(\ell+1)(n)^{m_{1\underline{\ell}(n)}}, \forall n.
$$
So acting with $T_{\underline \ell}^{m'+m''}$ we move colors from $\underline{\ell}$-row to $(\ell+1)$row. But regardless whether there is or there is not a factor in $u(\pi)$ with color in $\underline{\ell}$-row of $B_\ell$, we must not forget to move the $\underline{\ell}$-column below $\underline{\ell}\underline{\ell}$ to $(\ell+1)$-column, for which we need to act
$$
m'''=\sum_{n<0}\sum_{b=1}^{\ell-1}  m_{\underline{\ell}\underline{b}}(n)
$$
times with $T_{\underline \ell}$. Schematically we write
$$
T_{\underline \ell}^{m'''}\colon \ 
\underline{\ell}\underline{b}(n)^{m_{\underline{\ell}\underline{b}(n)}}, b<\ell,
\forall n
\rightsquigarrow
(\ell+1)\underline{b}(n)^{m_{\underline{\ell}\underline{b}(n)}}, b<\ell,
\forall n.
$$

After that we proceed acting with $T_{\underline a}^{m_{\underline a}}$ for 
$a=\ell-1, \dots, 2,1$, and arguing in a similar way.
\end{proof}
\begin{proof}[The proof of linear independence] Let 
\begin{equation} \label{E:LinDep}
\sum_{\pi\in S} C_\pi u(\pi)v_{k\Lambda_0}=0,
\end{equation}
for some nonempty set $S$ of $k\Lambda_0$-admissible monomials in $\bar{B}_{\ell}^{<0}$.
Let $\displaystyle M=\max_{\pi \in S} M(\pi)$. Let $\pi_0\in S$ be such that $M(\pi_0)=M$. 
Acting by $t(\pi_0)$ on \eqref{E:LinDep} one gets
$$\sum_{\pi\in S} C_\pi \left(T(\pi_0)u(\pi)\right) v_{k\Lambda_0}=0.$$
By using the pigeonhole principle, and an argument similar to the proof of 
Lemma \ref{L: action of powers Ta}, we see that
$$
T(\pi_0)u(\pi)=
\left\{ \begin{array}{ll}
C_\pi' w(\pi'),\ C_\pi'\in \mathbb C^\times& \text{if}\ t(\pi)=t(\pi_0),\\
0 & \text{otherwise}.  		
\end{array} \right.
$$
Hence,
$$\sum_{\pi\in S'} C_\pi' C_\pi w(\pi')v_{k\Lambda_0}=0,
 $$ for some nonempty subset $S'\subset S$. This is a relation of linear dependence between monomials from $\bar{B}_{2\ell}^{1,<0}$. Since these are, again, $k\Lambda_0$-admissible, and hence linearly independent, it follows that $C_\pi=0$ for $\pi \in S'$. In particular, $C_{\pi_0}=0$. By induction one gets $C_\pi=0$ for all $\pi \in S$, i.e. the set of monomial vectors $u(\pi) v_{k\Lambda_0}$ with $k\Lambda_0$-admissible $\pi$ is linearly independent, and hence a basis of $L(k\Lambda_0)$.
\end{proof}

\section*{Acknowledgement} This work has been partly supported by the Croatian Science Foundation under the project IP-2022-10-9006.


\end{document}